\documentclass{article}
\usepackage{latexsym,amssymb,amsmath,amsfonts,pictex,enumerate,amsthm,dsfont}
\usepackage{ mathrsfs }
\usepackage{tikz}
\usetikzlibrary{decorations.markings}
\usepackage{scalerel,stackengine}
\usepackage[active]{srcltx}

\theoremstyle{definition}
\newtheorem{theorem}{Theorem}
\newtheorem{definition}{Definition}
\newtheorem{lemma}[theorem]{Lemma}

\title{Geometric conditions for bounded point evaluations in spaces of several complex variables}
\author{Stephen Deterding\thanks{Department of Mathematics and Physics, Marshall University, Huntington WV 25705, USA}}
\date{}

\begin{document}

\maketitle

\begin{abstract}
    Let $U$ be a bounded domain in $\mathbb C^d$ and let $L^p_a(U)$, $1 \leq p < \infty$, denote the space of functions that are analytic on $\overline{U}$ and bounded in the $L^p$ norm on $U$. A point $x \in \overline{U}$ is said to be a bounded point evaluation for $L^p_a(U)$ if the linear functional $f \to f(x)$ is bounded in $L^p_a(U)$. In this paper, we provide a purely geometric condition given in terms of the Sobolev $q$-capacity for a point to be a bounded point evaluation for $L^p_a(U)$. This extends results known only for the single variable case to several complex variables.
\end{abstract}

\section{Introduction}

The theory of rational approximation of analytic functions in the complex plane has been well developed over many years since the publication of Runge's theorem in 1885; however, much less is known about rational approximation in the context of several complex variables, in part because Runge's theorem does not hold in higher dimensional spaces. There are many results known to be true in the one variable case for which it is either unknown whether they can be extended to a higher dimensional analogue or it is known that such an extension does not hold. In this paper, we consider one such result involving the existence of bounded point evaluations on the closure of rational functions in the $L^p$ norm. This corresponds to $L^p$ approximation as opposed to uniform approximation as in Runge's theorem.

Let $X \subseteq \mathbb C$ be a compact set with positive Lebesgue (area) measure, and let $R_0(X)$ denote the set of rational functions with poles off $X$. Then for $p \geq 1$, $R^p(X)$ is defined to be the closure of $R_0(X)$ in the $L^p(X)$ norm. 

\begin{definition}
A point $x \in X$ is said to be a \textbf{bounded point evaluation on} $R^p(X)$ if the linear functional $f \to f(x)$ defined for rational functions is bounded in $R^p(X)$. 
\end{definition}

Bounded point evaluations have been widely studied due to their connection with rational approximation and the invariant subspace problem. \cite{Brennan1971, Brennan2, Conway, Wolf}. Although bounded point evaluations are an analytic property of function spaces, in the one-dimensional case the existence of a bounded point evaluation can be characterized by purely geometric conditions on the underlying set. Specifically, for each type of approximation (e.g. uniform or $L^p$), there is an associated capacity connected to it and many properties of the approximation can be characterized by this capacity. For $R^p(X)$, the appropriate capacity is the Sobolev $q$-capacity, which can be defined more generally for subsets of $\mathbb C^d$ as opposed to other capacities, such as analytic capacity, which are only defined in the one variable case. This is helpful when trying to extend these one variable results to several complex variables as we will not need to determine the correct higher dimensional analogue for the appropriate capacity.

\begin{definition}
\label{q-capacity}
    Let $K \subseteq \mathbb C^d$ be a compact set. Then the \textbf{Sobolev $q$-capacity} is denoted by $\Gamma_q(K)$ and defined as

\begin{align*}
    \Gamma_q(K) = \inf_{u} \int |\nabla u|^q dV
\end{align*}

\noindent where the infimum is taken over all Lipschitz continuous functions $u$ with compact support and $u \geq 1$ on $K$.
    
\end{definition}

\noindent For a non-compact set $E$ the Sobolev $q$-capacity is defined as $\displaystyle \Gamma_q(E) = \sup_K \Gamma_q(K)$ where the supremum is taken over all compact subsets of $E$.

Let $a_n(x)$ denote the annulus $\{z\in \mathbb C: 2^{-(n+1)} < |z-x| < 2^{-n}\}$. The following theorem, which provides necessary and sufficient conditions for the existence of bounded point evaluations on $R^p(X)$, was first proven by Hedberg \cite{Hedberg1972} for the case of $p>2$ and Fernstr\"om and Polking \cite{Fernstrom} for $p=2$.

\begin{theorem}
\label{R^p}
Let $X \subseteq \mathbb C$ be a compact set. Then for $p \geq 2$, $x$ is a bounded point evaluation on $R^p(X)$ if and only if

\begin{align*}
    \sum_{n=1}^\infty 2^{nq} \Gamma_q(a_n(x) \setminus X) < \infty
\end{align*}

\noindent where $q = \frac{p}{p-1}$ and $\Gamma_q(E)$ denotes the Sobolev $q$-capacity of $E$. 

\end{theorem}

 The conditions in Theorem \ref{R^p} are purely geometric in that they only depend on the structure of the set $X$ and not on the properties of the functions in $R^p(X)$. If $1 \leq p < 2$ and $X$ has no interior, it is known that $R^p(X) = L^p(X)$ \cite{Havin, Sinanjan} and thus there are no bounded point evaluations on $R^p(X)$ \cite[Lemma 3.5]{Brennan1971}.

We now generalize these concepts to the case of several complex variables. Let $U$ be a bounded domain in $\mathbb C^d$ and let $O(\overline{U})$ denote the set of functions analytic on $\overline{U}$. Let $L^p_a(U)$ denote the closure of $O(\overline{U})$ in the $L^p$ norm. It follows from Runge's Theorem that if $d=1$, then $L^p_a(U) = R^p(\overline{U})$. 

\begin{definition}
A point $x \in \overline{U}$ is said to be a \textbf{bounded point evaluation on} $L^p_a(U)$ if the linear functional $f \to f(x)$ is bounded in $L^p_a(U)$.
\end{definition}

Bounded point evaluations have been studied in $\mathbb C^d$ due to their connections to both polynomial and rational approximations. \cite{ Lawrence, Range}. It follows directly from the definition that $x$ is a bounded point evaluation on $L^p_a(U)$ if and only if there exists a real constant $C>0$ such that 

\begin{align*}
    |f(x)| \leq C ||f||_{L^p(U)}
\end{align*}

\noindent for all $f \in L^p_a(U)$.

Our goal is to determine geometric conditions for the existence of bounded point evaluations on subsets of $\mathbb C^d$ when $d>1$. Similarly to the one variable case these conditions will be given in terms of an infinite sum of $q$-capacities; however, instead of annuli we will need to use a higher dimensional analogue of an annulus. Let $A_n(x) = \{z \in \mathbb C^d : 2^{-(n+1)} < |z-x| < 2^{-n} \}$. Our main result is the following.

\begin{theorem}
\label{main}
    Let $U$ be a bounded domain in $\mathbb C^d$ and let $p > \frac{2d}{2d-1}$. If

\begin{align}
\label{higher dimension series}
    \sum_{n=1}^\infty 2^{n(2d-1)q} \Gamma_q(A_n(x) \setminus U) < \infty
\end{align}

\noindent where $q = \frac{p}{p-1}$ and $\Gamma_q(E)$ denotes the Sobolev $q$-capacity of $E$, then $x$ is a bounded point evaluation for $L^p_a(U)$.
    
\end{theorem}

Key to the proof of Theorem \ref{main} is the use of the Bochner-Martinelli integral formula to represent evaluation at $x$. We first review this important theorem and some related results in the next section and then prove Theorem \ref{main} in Section 3.

\section{The Bochner-Martinelli Integral Formula}

To prove Theorem \ref{main} we will make use of the Bochner Martinelli integral formula to represent point evaluation at $x$ as an integral. We can then demonstrate that $x$ is a bounded point evaluation by obtaining a bound on this integral. 

\begin{definition}
The \textbf{Bochner-Martinelli kernel} is defined as

\begin{align*}
    w(\zeta,z) = \frac{(d-1)!}{(2\pi  i)^d} \frac{1}{|\zeta-z|^{2d}} \sum_{j=1}^d (\overline{\zeta_j} - \overline{z_j}) d\overline{\zeta_1} \wedge d\zeta_1 \wedge \ldots \wedge d\zeta_j \wedge \ldots \wedge d\overline{\zeta_d} \wedge d\zeta_d
\end{align*}

\noindent where the term $d\overline{\zeta_j}$ has been omitted from the wedge product. 

\end{definition}

\begin{theorem}[\textbf{Bochner-Martinelli Integral Formula}]
    Suppose that $f$ is analytic on the closure of a domain $E$ in $\mathbb C^d$ with piecewise smooth boundary $\partial E$. Then if $z\in E$,

\begin{align*}
    f(z) = \int_{\partial E} f(\zeta) w(\zeta,z)
\end{align*}

\noindent where $w(\zeta, z)$ is the Bochner-Martinelli kernel.
    
\end{theorem}

We conclude with the following lemma on the derivatives related to the Bochner-Martinelli kernel.

\begin{lemma}
    Suppose $\zeta \neq z$. Then

\begin{align*}
    \sum_{j=1}^d \frac{\partial}{\partial\overline{\zeta_j}} \left( \frac{\overline{\zeta_j}-\overline{z_j}}{|\zeta-z|^{2d}} \right) = 0
\end{align*}
    
\end{lemma}

\begin{proof}
To simplify calculations, we suppose that $z=0$. Since

\begin{align*}
    \frac{\partial}{\partial\overline{\zeta_j}} \left( \frac{\overline{\zeta_j}}{|\zeta|^{2d}} \right) = \frac{(\overline{\zeta_1}\zeta_1 + \ldots + \overline{\zeta_d} \zeta_d)^d-d(\overline{\zeta_1}\zeta_1 + \ldots + \overline{\zeta_d}\zeta_d)^{d-1} \overline{\zeta_j} \zeta_j}{|\zeta|^{4d}},
\end{align*}

\noindent it follows that

\begin{align*}
    \sum_{j=1}^d \frac{\partial}{\partial\overline{\zeta_j}} \left( \frac{\overline{\zeta_j}}{|\zeta|^{2d}} \right) = \frac{d(\overline{\zeta_1}\zeta_1 + \ldots + \overline{\zeta_d} \zeta_d)^d - d \sum_{j=1}^d (\overline{\zeta_1}\zeta_1 + \ldots + \overline{\zeta_d}\zeta_d)^{d-1} \overline{\zeta_j} \zeta_j}{|\zeta|^{4d}}.
\end{align*}

\noindent Using the multinomial theorem \cite[pg. 28]{Stanley} it can be shown that 

\begin{align*}
    \sum_{j=1}^d (\overline{\zeta_1}\zeta_1 + \ldots + \overline{\zeta_d}\zeta_d)^{d-1} \overline{\zeta_j} \zeta_j = (\overline{\zeta_1}\zeta_1 + \ldots + \overline{\zeta_d}\zeta_d)^d.
\end{align*}

\noindent Hence for $\zeta \neq 0$

\begin{align}
\label{vanishing derivative}
    \sum_{j=1}^d \frac{\partial}{\partial\overline{\zeta_j}} \left( \frac{\overline{\zeta_j}}{|\zeta|^{2d}} \right)=0
\end{align}

\noindent and the lemma is proved.  

\end{proof}

\section{Point Evaluations and Capacity}

The proof of Theorem \ref{main} proceeds in the following manner. The Bochner-Martinelli formula is used to obtain an integral representation for the point evaluation at $x$ and this integral is split up into subintegrals taken over nice sets containing the singularities of $f$ in $A_n$. Each subintegral is modified by multiplying the integrand by a Lipschitz function with nice properties that is identically $1$ on its domain of integration. The generalized Stokes' theorem and standard inequalities are then used to provide an upper bound for the integral in terms of the series of capacities \eqref{higher dimension series}, which shows that the point evaluation is bounded.

\begin{proof}
    Without loss of generality, we may suppose that $x=0$ and $U$ is contained in $B_1$, the unit ball in $\mathbb C^d$ centered at the origin. Let $f$ be analytic in a neighborhood of $\overline{U}$. We may also suppose that $f$ is modified so that it is continuous on $\mathbb C^d$ but still analytic on a neighborhood of $\overline{U}$. Let $A_n = \{ z\in \mathbb C^d : 2^{-(n+1)} < |z| < 2^{-n}\}$. We can find closed sets $K_n \subseteq \overline{A_n} \setminus \overline{U}$ with nice boundaries so that $f$ is analytic in $B_1 \setminus \bigcup K_n$ and $\int_{B_1 \setminus \cup K_n} |f|^p dV \leq 2 \int_U |f|^p dV$. It follows from the Bochner-Martinelli formula that

\begin{align*}
    f(0) &= \int_{\partial U} f(\zeta)w(\zeta,0)\\
    &= \sum_{n=1}^{\infty} \int_{\partial K_n} f(\zeta)w(\zeta,0)
\end{align*}

\noindent where $w(\zeta,z)$ is the Bochner-Martinelli kernel. Since $K_n \subseteq \overline{A_n} \setminus \overline{U}$, we can find Lipschitz functions $g_n$ with compact support so that $g_n \equiv 1$ on $K_n$ and $\int |\nabla g_n|^q dV \leq \Gamma_q(A_n\setminus U) + 2^{-2nd(2d-1)}$. We will show that such functions can be chosen to have support on $A_{n-1} \cup A_n \cup A_{n+1}$.

Let $\psi_n(z)$ be a piecewise linear continuous function of $|z|$ such that $0 \leq |\psi_n(z)| \leq 1$, $\psi_n(z) = 1$ on $A_n$, $\psi_n(0) = 0$ outside $A_{n-1} \cup A_n \cup A_{n+1}$ and $|\nabla \psi_n(z)| \leq  2^{n+2}$. Now let $\phi_n(z) = g_n(z)\psi_n(z) $. Then $\phi_n(z)$ has support on $A_{n-1} \cup A_n \cup A_{n+1}$ and

\begin{align*}
    \int |\nabla\phi_n(z)|^q dV \leq \int g_n(z)^q |\nabla \psi_n(z)|^q dV + \int \psi_n(z)^q |\nabla g_n(z)|^q dV. 
\end{align*}

\noindent Immediately we have that

\begin{align*}
    \int \psi_n(z)^q |\nabla g_n(z)|^q dV \leq \int |\nabla g_n(z)|^q dV \leq \Gamma_q(A_n \setminus U) + 2^{-2nd(2d-1)}.
\end{align*}

\noindent Let $q = \frac{p}{p-1}$. Since $p > \frac{2d}{2d-1}$ it follows that $q < 2d$ and hence H\"older's inequality with $p' = \frac{2d}{2d-q}$ and $q' = \frac{2d}{q}$ yields

\begin{align*}
    \int g_n(z)^q |\nabla \psi_n(z)|^q dV \leq \left(\int g_n(z)^{\frac{2dq}{2d-q}} dV\right)^{\frac{2d-q}{2d}} \left( \int |\nabla \psi_n(z)|^{2d} dV \right)^{\frac{q}{2d}}.
\end{align*}

\noindent By the General Sobolev inequality \cite[pg. 284]{Evans}

\begin{align*}
    \left(\int g_n(z)^{\frac{2dq}{2d-q}} dV\right)^{\frac{2d-q}{2d}} \leq C \int |\nabla g_n(z)|^q dV
\end{align*}

\noindent and since 

\begin{align*}
    \int |\nabla \psi_n(z)|^{2d} dV \leq C
\end{align*}

\noindent where $C$ is a constant only depending on $d$, it follows that

\begin{align*}
    \int |\nabla \phi_n(z)|^q dV \leq C \left( \Gamma_q(A_n\setminus U) +  2^{-2nd(2d-1)}\right)
\end{align*}

\noindent  as desired. Furthermore, $\phi_n(z) \equiv 1$ on $K_n$. Now let $\phi(z) = \sup_n \phi_n(z)$. Then $\phi(z) \equiv 1$ on $\bigcup_n K_n$ and 

\begin{align*}
    |\nabla \phi(z)| \leq \sum_{n=1}^\infty |\nabla \phi_n(z)|.
\end{align*}

\noindent Hence \begin{align*}
    f(0) = \sum_{n=1}^{\infty} \int_{\partial K_n} f(\zeta) \phi(\zeta)w(\zeta,0).
\end{align*}

\noindent From the generalized Stokes' theorem,

\begin{align*}
   \sum_{n=1}^{\infty} \int_{\partial K_n} f(\zeta)\phi(\zeta) w(\zeta,0) = \sum_{n=1}^\infty \int_{A_n \setminus K_n} \mathbf{d}[f(\zeta)\phi(\zeta)w(\zeta,0)]
\end{align*}

\noindent where $\mathbf{d}$ denotes the exterior derivative. Recall that $d\overline{\zeta_1}\wedge d\zeta_1 \wedge \ldots \wedge d\overline{\zeta_d} \wedge d\zeta_d = (2i)^d dV$. After computing the exterior derivative of the Bochner-Martinelli kernel and recalling that $d\zeta_j \wedge d\zeta_k = 0$ when $j=k$, we obtain

\begin{align*}    
   f(0) =\sum_{n=1}^\infty \int_{A_n \setminus K_n} \sum_{j=1}^d \frac{\partial}{\partial \overline{\zeta_j}} \left( f(\zeta) \phi(\zeta) \frac{(d-1)!}{\pi^d|\zeta|^{2d}} \overline{\zeta_j} \right) dV.
\end{align*}

\noindent Since $f$ is analytic off $K_n$ it follows that 

\begin{align*}
    \sum_{j=1}^d \frac{\partial}{\partial \overline{\zeta_j}} \left( f(\zeta) \phi(\zeta) \frac{(d-1)!}{\pi^d|\zeta|^{2d}} \overline{\zeta_j} \right) = \frac{f(\zeta) (d-1)!}{\pi^d} \sum_{j=1}^d \frac{\partial}{\partial \overline{\zeta_j}} \left( \phi(\zeta)\frac{\overline{\zeta_j}}{|\zeta|^{2d}} \right).
\end{align*}

\noindent Next, it follows from \eqref{vanishing derivative} that 

\begin{align*}
    \sum_{j=1}^d \frac{\partial}{\partial \overline{\zeta_j}} \left( \phi(\zeta)\frac{\overline{\zeta_j}}{|\zeta|^{2d}} \right) = \sum_{j=1}^d \frac{\partial \phi}{\partial \overline{\zeta_j}} \left(\frac{\overline{\zeta_j}}{|\zeta|^{2d}} \right)
\end{align*}

\noindent and thus 

\begin{align*}
    |f(0)| = \frac{(d-1)!}{\pi^d}\left| \sum_{n=1}^\infty \int_{A_n \setminus K_n} f(\zeta) \sum_{j=1}^d \frac{\partial \phi}{\partial \overline{\zeta_j}} \left(\frac{\overline{\zeta_j}}{|\zeta|^{2d}}\right) dV\right|.
\end{align*}

\noindent Applying the Cauchy-Schwarz inequality yields  

\begin{align*}
    |f(0)| &\leq C \sum_{n=1}^\infty \int_{A_n \setminus K_n} \frac{|f(\zeta)|}{|\zeta|^{2d-1}}|\nabla \phi(\zeta)| dV\\
    &\leq C \sum_{n=1}^\infty 2^{n(2d-1)} \int_{A_n \setminus K_n} |f(\zeta)| \left(|\nabla \phi_{n-1}(\zeta)| + |\nabla \phi_{n}(\zeta)| + |\nabla \phi_{n+1}(\zeta)| \right) dV.
\end{align*}

\noindent This sum can be bounded using H\"older's inequality and \eqref{higher dimension series} as follows.

\begin{align*}
    |f(0)| &\leq C \sum_{n=1}^\infty 2^{n(2d-1)} \left( \int_{A_n \setminus K_n} |f(\zeta)|^p dV \right)^{\frac{1}{p}} \left( \int_{A_n \setminus K_n} \left(|\nabla \phi_{n-1}(\zeta)| + |\nabla \phi_{n}(\zeta)| + |\nabla \phi_{n+1}(\zeta)| \right)^q dV \right)^{\frac{1}{q}}\\
    &\leq C \sum_{n=1}^\infty 2^{n(2d-1)} \left( \int_{A_n \setminus K_n} |f(\zeta)|^p dV \right)^{\frac{1}{p}} \left( \Gamma_q(A_{n-1} \setminus U) + \Gamma_q(A_n \setminus U) + \Gamma_q(A_{n+1} \setminus U) + 2^{-2nd(2d-1)} \right)^{\frac{1}{q}}\\
    &\leq C \left( \sum_{n=1}^{\infty} \int_{A_n \setminus K_n} |f(\zeta)|^p dV \right)^{\frac{1}{p}} \left( \sum_{n=1}^{\infty} 2^{n(2d-1)q} \Gamma_q(A_n \setminus U) + C \right)^{\frac{1}{q}}\\
    &\leq C ||f||_p
\end{align*}

\noindent Thus $0$ is a bounded point evaluation for $L^p_a(U)$.
    
\end{proof}

\bibliographystyle{abbrv}
\bibliography{biblio}

\end{document}